\documentclass[12pt]{amsart}
\usepackage{amsmath}
\usepackage{amsfonts}
\usepackage{amsthm}
\usepackage{amssymb}
\usepackage{amscd}
\usepackage[all]{xy}
\usepackage{enumerate}
\usepackage{hyperref}
\usepackage{comment}
\usepackage{paralist}
\usepackage{tikz}
\usetikzlibrary{cd}
\usepackage{mathtools}
\usepackage[latin1]{inputenc}
\usepackage{mathrsfs}
\usepackage{array}

\keywords{} 

\subjclass[2010]{}

\theoremstyle{plain}
\newtheorem{thm}{Theorem}[section]

\newtheorem{prop}[thm]{Proposition}

\newtheorem{cor}[thm]{Corollary}

\newtheorem{lem}[thm]{Lemma}
\newtheorem{cla}[thm]{Claim}

\theoremstyle{definition}

\newtheorem{rmk}[thm]{Remark}

\newcommand{\sO}{\mathcal{O}}

\newcommand{\sZ}{\mathcal{Z}}

\newcommand{\mD}{\mathbb{D}}

\newcommand{\mG}{\mathbb{G}}

\newcommand{\mP}{\mathbb{P}}

\newcommand{\mZ}{\mathbb{Z}}

\newcommand{\Pic}{\mathrm{Pic}\,}

\usepackage{color}

\def\geq{\geqslant}
\def\leq{\leqslant}

\numberwithin{equation}{section}

\newcommand{\beba}  {\begin{equation}\begin{array}{rcl}}

\newcommand{\eaee}  {\end{array}\end{equation}}

\makeatletter
\def\l@section{\@tocline{1}{0pt}{1pc}{}{}}
\def\l@subsection{\@tocline{2}{0pt}{1pc}{4.6em}{}}
\def\l@subsection{\@tocline{3}{0pt}{1pc}{7.6em}{}}
\renewcommand{\tocsection}[3]{%
  \indentlabel{\@ifnotempty{#2}{\makebox[2.3em][l]{%
    \ignorespaces#1 #2.\hfill}}}#3}
\renewcommand{\tocsubsection}[3]{%
  \indentlabel{\@ifnotempty{#2}{\hspace*{2.3em}\makebox[2.3em][l]{%
    \ignorespaces#1 #2.\hfill}}}#3}
\renewcommand{\tocsubsection}[3]{%
  \indentlabel{\@ifnotempty{#2}{\hspace*{4.6em}\makebox[3em][l]{%
    \ignorespaces#1 #2.\hfill}}}#3}
\makeatother

\title[On the variety of tritangential planes]{On the variety of tritangential planes to a general K3 surface of degree 6 and genus 4 in $\mP^4$}

\author{Ciro Ciliberto}
\address{
Ciro Ciliberto\\Dipartimento di Matematica\\
Universit\`a di Roma ``Tor Vergata''\\
Via della Ricerca  Scientifica, 00177 Roma\\ Italia
\texttt{cilibert@axp.mat.uniroma2.it   }}

\author{Sandro Verra}
\address{Sandro Verra\\Departimento di Matematica e  Fisica \\
Universit\`a Roma 3\\
Largo San Leonardo Murialdo,  00146 Roma \\ Italia
\texttt{sandro.verra@gmail.com}}

\begin{document}

\markboth{}{}
	\label{cast2}
\maketitle

\begin{abstract} {Let $S\subset \mP^4$ be a general K3 surface of degree 6 and genus 4. In this paper we study the irreducible variety $X_S$ of \emph{tritangential planes} to $S$ whose general point is a plane that intersects $S$ in a curvilinear scheme of length six supported at three non collinear points. The variety $X_S$ can be identified as the relevant part of the fixed locus of the so called \emph{Beauville involution} defined on the Hilbert scheme $S[3]$ of 0--dimensional schemes of length three of $S$. In this paper we prove that: (a) $X_S$ has dimension 3,  is irreducible and smooth, except for 210 points that are at most of multiplicity 2  for $X_S$;  (b) $X_S$, in its natural embedding in the Grassmannian $\mathbb G(2,4)\subset \mP^9$ of planes in $\mP^4$, has degree $152$.} \end{abstract}

\tableofcontents

\section{Introduction}

Let $S\subset \mathbb P^g$ be a smooth, irreducible, linearly normal, complex K3 surface of degree $2g-2$ and genus $g$. 
In this paper we will  assume that $S$ is \emph{general},  so that in particular $\Pic(S)=\mZ\langle \sO_S(1)\rangle$.  We will denote by $S[h]$ the Hilbert scheme of $0$--dimensional subschemes of length $h$ of $S$, that is a smooth, irreducible, projective  variety of dimension $2h$. 

There is a birational involution $\beta_S: S[g-1]\dasharrow S[g-1]$, called the \emph{Beauville involution}, defined in the following way. It maps the general point $Z\in S[g-1]$ to the scheme $Z'\in  S[g-1]$ such that the union of $Z$ and $Z'$ is the scheme theoretic intersection of $S$ with the linear subspace of dimension $g-2$ spanned by $Z$. The indeterminacy locus ${\rm Ind}(\beta_S)$ of $\beta_S$ consists of the set of those schemes $Z\in S[g-1]$ such that the span of $Z$ has dimension smaller than $g-2$.

There is also a rational map ${\rm sp}_S: S[g-1]\dasharrow \mG(g-2,g)$ (where $\mG(g-2,g)$ is the Grassmannian of $(g-2)$--dimensional linear subspaces of $\mP^g$) that maps the general $Z\in S[g-1]$ to the linear subspace of dimension $g-2$ spanned by $Z$. The map ${\rm sp}_S$ is dominant, has degree ${2g-2}\choose {g-2}$, and it has the same indeterminacy locus ${\rm Ind}(\beta_S)$ as $\beta_S$.
Of course ${\rm sp}_S(Z)={\rm sp}_S(\beta_S(Z))$, whenever $Z\in S[g-1]\setminus {\rm Ind}(\beta_S)$. 

The involution $\beta_S$ has been widely studied in the case $g=3$, in which the indeterminacy locus is empty, mostly in relation with the locus of its fixed points, that  identifies with the set of \emph{bitangent lines} to a  general quartic surface $S\subset \mP^3$ (see, e.g., \cite {CZ2, CZ3, W}).
By contrast little is known about the  involution $\beta_S$ for $g>3$. Here we will focus on the case $g=4$, in which $S$ is a general complete intersection surface of a smooth quadric $\mathbb Q$ and a smooth cubic $\mathbb F$  hypersurfaces in $\mP^4$ (note that $\mathbb Q$ is uniquely determined by $S$, whereas $\mathbb F$ is not). 

In this case the indeterminacy set ${\rm Ind}(\beta_S)$ of $\beta_S$ (and of ${\rm sp}_S$) is the set of schemes $Z\in S[3]$ such that $\langle Z\rangle$ is a line: we call them \emph{triples of collinear points}. The set ${\rm Ind}(\beta_S)$ is easy to identify. Indeed $Z\in {\rm Ind}(\beta_S)$ if and only if the line $\langle Z\rangle$ is contained in the unique quadric $\mathbb Q$ containing $S$, so that there is a natural  isomorphism   between ${\rm Ind}(\beta_S)$ and the Hilbert scheme of lines in $\mathbb Q$, that is isomorphic to $\mP^3$. 

To resolve the indeterminacies of $\beta_S$ (and of ${\rm sp}_S$) one must blow--up ${\rm Ind}(\beta_S)$ in $S[3]$. Let us consider $p_S: \widetilde {S[3]} \to S[3]$ this blow-up and let $E_S$ be the exceptional divisor that sits over ${\rm Ind}(\beta_S)$. Then $E_S$, that has dimension 5, is isomorphic to the set of pairs $(Z,\pi)$, with $Z\in {\rm Ind}(\beta_S)$ and $\pi$ a plane containing the line $\langle Z\rangle$. 

The map ${\rm sp}_S$ extends to a morphism $\widetilde {{\rm sp}_S}: \widetilde {S[3]}\longrightarrow \mG(2,4)$ that acts as ${\rm sp}_S$ off $E_S$ and on $E_S$ it acts by sending a pair $(Z,\pi)\in E_S$ to the plane $\pi$.

The map $\beta_S$, in turn, extends to an involution morphism $\widetilde{\beta_S}: \widetilde {S[3]} \longrightarrow \widetilde {S[3]}$, that coincides with $\beta_S$ off $E_S$ and on $E_S$ works by mapping a pair $(Z,\pi)\in E_S$ to the pair $(Z',\pi)\in E_S$ such that the union of $Z$ and $Z'$ is the intersection scheme of $S$ and $\pi$, so that $\widetilde{\beta_S}$ induces an involution on $E_S$. 

In this paper we study the variety of fixed points of $\widetilde{\beta_S}$. This variety consists of two disjoint irreducible components (see \S \ref {sec:Beau} and \S \ref {ssec:irred}). One component is the set $Y_S$ of fixed points of the restriction of $\widetilde{\beta_S}$ on $E_S$, which can be easily described (see \S \ref {sec:Beau}). The other component $X_S$, which is more important to us, has the property that any point $Z\in X_S$ is such that $Z$ spans a plane and the general point $Z$ of $X_S$ consists of three distinct points and is such that the plane $\langle Z\rangle$ intersects $S$ in a curvilinear scheme of length six supported on $Z$. We call such a plane a \emph{3--tangential plane} to $S$ at $Z$, and we will identify $X_S$ with its image via ${\rm sp}_S$, that is injective on $X_S$. Hence $X_S$ can be seen as  the closure in the Grassmannian $\mathbb G(2,4)$ of the set of tritangential planes to $S$, and therefore it is called the \emph{variety of tritangential planes} to $S$. This is the main character in this paper. 

In \S\S \ref {ssec:local} and \ref {ssec:special} we prove that $X_S$ is smooth, except may be for the points of a 0--dimensional scheme of length 210, corresponding to the so called \emph{special planes} for $S$, i.e., those tangent planes to $S$ that intersect $S$ uniquely at their tangency point with $S$. These points have multiplicity at most 2 for $X_S$.  In order to compute the number 210 we make in Proposition \ref {prop:exp} a more general computation, i.e., given a smooth, irreducible, projective surface in $\mP^4$, we compute the (expected) number of hyperplane sections of the surface with a triple point. 

In \S \ref {sec:degre} we compute the degree, equal to 152, of $X_S$ considered as a subvariety of $\mathbb G(2,4)\subset \mP^9$. The main point of this computation is the computation of the degree of the scroll of (proper) tritangent lines to a general projection in $\mP^3$ of a general K3 surface of degree 6 and genus 4 in $\mP^4$.

We suspect that the variety $X_S$ of tritangential planes a general K3 surface of degree 6 and genus 4 in $\mP^4$ is of general type, but we have been unable to prove this. 

It would be highly desirable to have a description as the one we make here for the fixed locus of $\beta_S$ when $S$ is a general K3 surface of degree 6 and genus 4 in $\mP^4$, also, more generally, for a general K3 surface $S$ of degree $2g-2$ and genus g in $\mP^g$, for $g\geq 5$. This is however much harder. First of all it is in general  more intricate the study of the indeterminacy locus ${\rm Ind}(\beta_S)$ of $\beta_S$. The analogue of the variety of tritangential planes in the case of genus 4, is the closure $X_S$ of the locus of \emph{$(g-1)$--tangential} linear spaces of dimension $g-2$ to $S$, i.e., of linear spaces of dimension  $g-2$ that intersect $S$ in a curvilinear scheme of degree $2g-2$ supported on a scheme $Z$ consisting of $g-1$ distinct points. It is very likely that this $X_S$ is also irreducible and generically smooth, but studying its singular locus is definitely more complicated. Moreover the computation of the degree of $X_S$ seems  in general very hard, as well as the study of its birational structure, though we believe that $X_S$ should always be of general type.  We hope to come back on these topics in the future. \medskip

\noindent {\bf Aknowledgement}:  The authors are members of GNSAGA of the Istituto Nazionale di Alta Matematica ``F. Severi''. \color{black}
\medskip

\section{The fixed locus of the Beauville involution for $g=4$}\label{sec:Beau}

Let $S\subset \mP^4$ be a general K3 surface of degree 6 and genus 4.  We want to study the set of fixed points of the  Beauville  involution $\widetilde{\beta_S}: \widetilde {S[3]} \longrightarrow \widetilde {S[3]}$. First of all there is a set of fixed points on $E_S$. Precisely, $(Z,\pi)\in E_S$ is a fixed point for $\widetilde{\beta_S}$ if and only if the plane $\pi$ is tangent along a line to the unique quadric  $\mathbb Q$ containing $S$, i.e., it intersects the quadric $\mathbb Q$ along a line with multiplicity 2. For each line $\ell\subset \mathbb Q$ there is a unique plane $\pi_\ell$ that is tangent to $\mathbb Q$ along $\ell$, namely the intersection of all hyperplanes tangent to $\mathbb Q$ at points of $\ell$, hyperplanes that form a pencil with base locus the plane in question. Hence the set $Y_S$  of fixed points of $\widetilde{\beta_S}$  along $E_S$ is naturally isomorphic to the Hilbert scheme of lines in $\mathbb Q$, that is in turn isomorphic to $\mP^3$. 

As mentioned in the Introduction, the remaining set of fixed points is denoted by $X_S$.

\begin{prop}\label{prop:noint} In the above set up one has $X_S\cap Y_S=\emptyset$.
\end{prop}

\begin{proof} Suppose, by contradiction, that $X_S\cap Y_S\neq \emptyset$ and let $(Z,\pi)\in X_S\cap Y_S$. We can then find a flat family $\sZ\to \mD$, where $\mD$ is a disc, such that for $t\in \mD\setminus \{0\}$ the fibre of $\sZ\to \mD$ over $t$ is a scheme $Z_t\in S[3]$ such that $\pi_t:=\langle Z_t\rangle$ is a 3--tangential plane to $S$ at $Z_t$, whereas  the fibre of $\sZ\to \mD$ over $0$ is $Z_0=Z$, and the flat limit of $\pi_t$, when $t$ tends to $0$, is $\pi_0=\pi$.  

Fix a general point $p\in \mP^4$ and consider for all $t\in \mD$, the hyperplane $H_t:=\langle p,\pi_t\rangle$. Let $C_t$ be the curve cut out by $H_t$ on $S$. By the generality of $p$ we may assume that the curve $C_0$ is smooth and therefore, by may be shrinking $\mD$, we may assume that $C_t$ is smooth for all $t\in \mD$, so that $C_t$ is a canonical curve of genus 4   for all $t\in \mD$. Now for all $t\in \mD$, $\sO_{C_t}(Z_t)$ is a theta--characteristic on $C_t$. For $t\neq 0$, $\sO_{C_t}(Z_t)$ is an odd theta--characteristic, whereas $\sO_{C_0}(Z_0)$ is an even theta--characteristic, and this is impossible by the constancy of parity of theta--characteristics in flat families of curves (see \cite [Appendix B, \S 33] {ACGH}). This proves the assertion. \end{proof}

By Proposition \ref {prop:noint},  any point $Z\in X_S$ is such that $Z$ spans a plane, and the map  ${\rm sp}_S$ is injective on $X_S$. Thus, in the rest of this paper, we will identify
$X_S$  with its image in $\mG(2,4)$. 

\section{Irreducibility and local behaviour}

\subsection{Irreducibility}\label{ssec:irred} 
Let $\mathcal M_g$ be the moduli space of curves of genus $g$ and consider the variety $\mathcal S^-_g $ of all pairs $(C,\theta)$ where $C$ is a smooth irreducible curve of genus $g$ and $\theta$ an odd theta--characteristic on $C$. We have the well known forgetful \'etale morphism $\mathcal S^-_g \longrightarrow \mathcal M_g$ of degree $ 2^{ g - 1} ( 2 ^g - 1 )$ and therefore all irreducible components of $\mathcal S^-_g $ have dimension $3g-3$. We will need the following result:

\begin{thm}[See \cite {Co}, Lemma (6.3)] \label{thm:irre} The variety $\mathcal S^-_g $ is irreducible for all $g\geq 2$. In particular the monodromy of the covering $\mathcal S^-_g \longrightarrow \mathcal M_g$ is transitive for $g\geq 2$. 
\end{thm}

 Next we keep the set up of  \S \ref {sec:Beau}.

\begin{thm} \label{thm:irr} If $S$ is a general K3 surface of degree 6 and genus 4 in $\mP^4$, then  $X_S$ is irreducible of dimension 3.
\end{thm}

\begin{proof} Let us consider the variety $\mathcal X_S$ of all pairs $(\pi,H)\in X_S\times (\mP^4)^\vee$ such that  $\pi\subset H$. We have the obvious projection
$$
p_2: (\pi,H)\in \mathcal X_S\longrightarrow H\in (\mP^4)^\vee.
$$
This is a finite morphism of degree 120, and this proves that $ \mathcal X_S$ has dimension 4. Look now at the other projection
$$
p_1: (\pi,H)\in \mathcal X_S\longrightarrow \pi\in X_S.
$$
All fibres of $p_1$ are isomorphic to $\mP^1$. This shows that $X_S$ has dimension 3. 

To prove irreducibility, it suffices to show that $\mathcal X_S$ is irreducible. To see this, notice that, since $S$ is general, the general hyperplane section of $S$ is a general curve of genus $4$. From Theorem \ref {thm:irre} it follows that the monodromy of $p_2$ is transitive, and this implies the irreducibility of $\mathcal X_S$, as wanted.  
\end{proof}

\subsection{Regular and special planes} Let $S$ be a general K3 surface of degree 6 and genus 4 in $\mP^4$, and let $\pi\in X_S$. Then $\pi$ comes equipped with a $0$--dimensional length 3 scheme $Z$ such that the $0$--dimensional length 6 intersection scheme $Z'$ of $\pi$ and $S$ has the same support as $Z$ and at every point of $Z$ has lenght at least 2.  If $Z$ is curvilinear, we will say that $\pi$ is \emph{regular} for $S$. In this case the general hyperplane containing $\pi$ cuts out on $S$, by Bertini's theorem,  a smooth curve $C$, and the base point locus of the pencil of curves cut out on $S$ by the pencil $\mathcal H_\pi$ of hyperplanes containing $\pi$ has base locus scheme the curvilinear 
$0$--dimensional length 6 scheme $Z'$ that has the same support as $Z$. 

 The plane  $\pi\in X_S$ is called \emph{special} for $S$ if it is not regular.  This is the case if and only if the $0$--dimensional length 3 scheme $Z$ is not curvilinear,  hence there exists an intersection point $p$ of $\pi$ with $S$ such that the defining ideal of $Z$ on $\pi$ is $\mathfrak m_p^2$, where $\mathfrak m_p$ is the maximal ideal of $p$ on $\pi$. In particular $\pi$ is tangent to $S$ at $p$,  the $0$--dimensional length 6 intersection scheme $Z'$ of $\pi$ and $S$ is supported at $p$ and contains $Z$, and $\pi$ intersects $S$ only at $p$ with intersection multiplicity 6. 

\begin{prop}\label{prop:spec} Let $S$ be a general K3 surface of degree 6 and genus 4 in $\mP^4$. Then a plane $\pi$ of $\mP^4$ is special for $S$ if and only if there is a point  $p\in S$ such that $\pi$ is tangent to $S$ at $p$ and there is  a hyperplane containing $\pi$ that cuts out on $S$ a curve with a triple point at $p$. \end{prop}

\begin{proof}  Suppose that $\pi\in X_S$ is special. Then $\pi$ is tangent to $S$ at a point $p$ and $\pi$ intersects $S$ only at $p$ with intersection multiplicity 6.  Consider the pencil of curves $\mathcal P_\pi$ cut out on $S$ by the pencil $\mathcal H_\pi$ of hyperplanes containing $\pi$. The base locus scheme of $\mathcal P_\pi$ is supported at $p$ and all curves of $\mathcal P_\pi$ are singular at $p$. Let $C$ be a general curve in $\mathcal P_\pi$. It has a point of multiplicity exactly 2 at $p$, because otherwise two general curves of $\mathcal P_\pi$ would have intersection multiplicity at least 9 at $p$, a contradiction. Let $C'$ be another general curve in $\mathcal P_\pi$. Then $C$ and $C'$ have intersection multiplicity 6 at $p$. Then only the following two possibilities may occur:\\ \begin{inparaenum}
\item [$\bullet$] $C$ and $C'$ have both a node at $p$ and the two branches of $C$ at $p$ are tangent to the two branches of $C'$ at $p$;\\
\item [$\bullet$] $C$ and $C'$ have both a simple cusp at $p$ with the same cuspidal tangent.
\end{inparaenum}

In either case, by imposing one single condition to the curves in $\mathcal P_\pi$ we can obtain a curve $C_0$ in $\mathcal P$ with multiplicity at least 3 at $p$. The multiplicity of $C_0$ at $p$ has to be exactly 3 because otherwise $C_0$ would have intersection multiplicity at least 8 at $p$ with the general curve $C$ of $\mathcal P_\pi$, a contradiction. 

Conversely, suppose we have a hyperplane $H$ that cuts out on $S$ a curve $C_0$ with a triple point $p$. Let $\pi$ be the tangent plane to $S$ at $p$, that is contained in $H$. Consider the pencil $\mathcal P_\pi$ of curves cut out on $S$ by the pencil $\mathcal H_\pi$ of hyperplanes containing $\pi$. Then $C_0$ belongs to $\mathcal P_\pi$. The general curve $C$ in $\mathcal P_\pi$ is singular in $p$, so it must have a double point at $p$. Hence the base locus scheme of $\mathcal P_\pi$ is supported at $p$, and this implies that $\pi$ intersects $S$ only at $p$ with intersection multiplicity 6. This yields that $\pi$ is special for $S$. 
\end{proof}

\subsection{Local behaviour}\label{ssec:local}  Let again $S$ be a general K3 surface of degree 6 and genus 4 in $\mP^4$.  In this section we study the local structure of the points of $X_S$. 

\begin{thm}\label{thm:local} If $\pi\in X_S$ is regular, then it is a smooth point for $X_S$. If $\pi\in X_S$ is special then it is a  point of multiplicity at most 2 for $X_S$.
\end{thm}

\begin{proof} Let us consider again the variety $\mathcal X_S$ introduced in the proof of Theorem \ref {thm:irr}, with the projection $p_2: (\pi,H)\in \mathcal X_S\longrightarrow H\in (\mP^4)^\vee$, that is a finite morphism of degree 120. Let $\pi\in X_S$ be regular. Then we can choose a hyperplane $H$ containing $\pi$ such that the curve $C$ cut out by $H$ on $S$ is smooth and $\pi$ cuts out on $C$ an odd theta--characteristic. The fibre $p_2^{-1} (H)$ consists of the points of the form $(\pi',H)$, with $\pi'$ cutting out on $C$ an odd theta--characteristic. This implies that the fibre $p_2^{-1} (H)$ consists of exactly 120 distinct points, hence $p_2$ is \'etale over $H$. This implies that $\mathcal X_S$ is smooth at each point of the fibre $p_2^{-1} (H)$, in particular it is smooth at 
$(\pi,H)$. By looking at the projection $p_1: (\pi,H)\in \mathcal X_S\longrightarrow \pi\in X_S$, we see that  $\mathcal X_S$ is a $\mP^1$ bundle over $X_S$. So that, being $\mathcal X_S$ smooth at $(\pi, H)$, implies that $X_S$ is smooth at $\pi$.

Let us assume next that $\pi\in X_S$ is special. Then we can choose a hyperplane $H$ containing $\pi$ such that the curve $C$ cut out by $H$ on $S$ has a node or a simple cusp at a point $p$ and it is otherwise smooth. Moreover the plane $\pi$ intersects $C$ only at $p$. Then the fibre $p_2^{-1} (H)$ consists of the points of the form $(\pi',H)$, with $\pi'$ a
plane whose intersection with $C$ consists of 3 (proper or infinitely near) points each counted with multiplicty 2. Such planes are of two types:\\ \begin{inparaenum}
\item [(i)] planes $\pi'$ that are tangent at three (proper or infinitely near) smooth points of $C$;\\
\item [(ii)] planes $\pi'$  passing through the double point $p$ of $C$ and tangent at two (proper or infinitely near) smooth points of $C$. 
\end{inparaenum}

The planes of type (i) are exactly 64 (by  \cite [Cor. (2.7) and (2.13)]{H}). As for the planes of type (ii), by projecting from $p$ we get a smooth plane quartic $C_0$ and the planes in question project to bitangent lines to $C_0$, which are 28. We claim that each of these planes has to be counted with multiplicity two. Let us give this for granted for the time being. Then the fibre  $p_2^{-1} (H)$ contains the point $(\pi, H)$ with multiplicity 2, and this implies that 
$(\pi, H)$ is a point of multiplicity at most 2 for $\mathcal X_S$. Then, arguing as above, we see that $\pi$ is a point of multiplicity at most 2 for $X_S$. 

To finish our proof we have to prove the above claim. To do this, we note that $C$ is a limit of a general nodal hyperplane section $C'$ of $S$. Let $H'$ be the hyperplane spanned by $C'$ and let $q$ be the node of $C'$. Then  the fibre $p_2^{-1} (H')$ consists of the points of the form $(\pi',H')$, with:\\ \begin{inparaenum}
\item [(i')]  $\pi'$ a plane that is tangent at three (proper or infinitely near) smooth points of $C'$;\\
\item [(ii')] $\pi'$  a plane passing through the node $q$ of $C$ and tangent at two (proper or infinitely near) smooth points of $C$. 
\end{inparaenum}

The planes of type (i') are again 64. By projecting from $q$ we get a general smooth plane quartic $C'_0$ and the planes of type (ii') project to bitangents to $C'_0$, which are 28. By generality, we may assume that these bitangents are acted upon by a transitive monodromy. This shows that the planes of type (ii') have to be counted with the same multiplicity $m$. Since $64+28m=120$, we see that $m=2$, as wanted.  \end{proof}

\subsection{Enumeration of the special planes}  \label {ssec:special}
To prove the next result we need a preliminary fact that we cite from \cite {Gal}. Let $\mathbb D$ be a disk and let $\mathcal X\longrightarrow \mathbb D$ be a flat family of projective surfaces with smooth total space and smooth, irreducible general fiber $\mathcal X_t$, with $t\in \mathbb D\setminus \{0\}$, and whose special fiber $\mathcal X_0 = A\cup B$ consists  of the union of two irreducible, reduced components $A$ and $B$, intersecting transversally along a smooth curve $R = A\cap B$. Let $\mathcal L$ be a line bundle on $\mathcal X$.  Let $\mathcal C\longrightarrow \mathbb D$ be a flat family of curves, with $\mathcal C$ a section of $\mathcal L$ on $\mathcal X$. We denote by $C_t$ the fibre of $\mathcal C\longrightarrow \mathbb D$ over $t\in \mathbb D$. We assume that, for  $t\in \mathbb D\setminus \{0\}$, $C_t$ is reduced,
with a triple point $p_t$. We assume, moreover, that the curve $C_t$ tends to the curve $C_ 0$ in such a way that $p_ t$ tends to a point $p \in R$. We will set $C_0=C_A\cup C_B$ with $C_A$ and $C_B$ the components of $C_0$ on $A$ and $B$ respectively. The result we need is the following:

\begin{prop}[\cite {Gal}, \S 3] \label{prop:lim}
 In the above set up $C_0$ has at least the following singularity: both $C_A$ and $C_B$ have a double point at $p$ with one branch tangent to $R$.
\end{prop}

The singularity of $C_0$ at $p$ described in Proposition \ref {prop:lim} will be called a \emph{limit triple point}.  Now we are in a position to prove the following:

\begin{prop}\label{prop:fin} Let $S$ be a general K3 surface of degree 6 and genus 4 in $\mP^4$. There are only finitely many hyperplane sections of $S$ having a triple point. In particular the set of special planes for $S$ is finite. 
\end{prop} 

\begin{proof} Let $X$ be a general hypersurface of $\mathbb P^4$ of degree $3$ with equation $f=0$, let $H$, $H'$ be general hyperplanes in $\mP^4$ with equations $\ell=0$ and $\ell'=0$ respectively, and let $Y$ be a general hypersurface of degree 2 in $\mP^4$ with equation $g=0$. We  consider the threefold $\mathcal X''$ in  $X\times \mathbb A^1\subset  \mathbb P^4\times \mathbb A^1 $ defined in  $\mathbb P^4\times \mathbb A^1 $ by the equations $\{f=tg+\ell\ell' =0,\,\text{with}\,t\in \mathbb A^1\}$. Via the second projection $\mathcal X''\longrightarrow\mathbb A^1$ this is a flat family of surfaces, with smooth general fibre $\mathcal X''_t$, that is a general
complete intersection surface of type $(2,3)$ in $\mP^4$, and  whose fibre over $0$ is $\mathcal X''_0=A\cup A'$ with $A=X\cap H, A'= X\cap H'$ two general cubic surfaces in $\mP^3$, and $A\cap A'=R$ is a smooth plane cubic curve, section of $A$ and $A'$ with the plane $\pi$ with equations $\ell=\ell'=0$.

 We are interested in the singularities of $\mathcal X''$ in a neighbordhood of the central fibre, i.e., we are interested in what happens if $t$ belongs to a disc $\mathbb D$, centered at the origin of $\mathbb A^1$. Thus we consider the flat family $\mathcal X'=\{f=tg+\ell\ell' =0,\,{ \rm with}\,t\in \mathbb D\}\to\mathbb D$.  It is immediate to see
that the singular locus of $\mathcal X'$ coincides with the scheme with equation $t=f=g=\ell=\ell'=0\subset \mathcal X'_0$, that consists of six distinct points $p_1,\ldots, p_6$, located on the plane $\pi$ and complete intersection there of the cubic $R$ with the quadric $Y$. Moreover, $\mathcal X'$ has points of type $A_1$ at $p_1,\ldots, p_6$.
We resolve these singularities by blowing-up $\mathcal X'$ at $p_1,\ldots, p_6$. One obtains a new family $\tilde{\mathcal X}\longrightarrow\mathbb D$ with the same general fibre as $\mathcal X'\longrightarrow\mathbb D$
 and whose central fibre consists of eight components $\tilde A$ and $\tilde A'$, the blow-ups of $A$ and of $A'$ along $p_1,\ldots, p_6$, plus the six exceptional divisors $E_i$ over $p_i$, for $\leq i\leq 6$. Each of these exceptional divisors is  
 isomorphic to $\mathbb P^1\times \mathbb P^1$. Now we can contract each of these exceptional divisors by contracting one of the two rulings of   
 $\mathbb P^1\times \mathbb P^1$. We choose to do this in the direction of $\tilde A'$.
 We thus obtain a new family of 3-folds $\mathcal X\to\mathbb D$, with $\mathcal X$ smooth,  with fiber $\mathcal X_t=\mathcal X'_t$ over $t\neq 0$, and  whose central fiber $\mathcal X_0=A\cup B$, where  $B=\tilde A'$. Moreover 
 $A$ and $B$ intersect transversally along a curve that is isomorphic to  $R$.
 
 There is an obvious morphism $\mathcal X\longrightarrow  \mP^4$, and we denote by $\mathcal L$ the pull back of $\mathcal O_{\mP^4}(1)$ via this morphism. The restriction  $\mathcal L_t$ of $\mathcal L$ to $\mathcal X_t$, with $t\neq 0$, coincides with $\mathcal O_{\mathcal X_t}(1)$. The restriction  $\mathcal L_0$ of $\mathcal L$ to $\mathcal X_0$ maps $A\cup B$ to $A\cup A'$. 
 
Suppose now,  by contradiction, that for $t\in \mathbb D$ general, the surface $\mathcal X_t$ has infinitely many hyperplane sections with a triple point. In particular, there would be a 1--dimensional family $\mathcal H_t$ of such hyperplane sections of $\mathcal X_t$.  Then  also $\mathcal X_0$ would have a 1--dimensional family $\mathcal H_0$ of 
 sections in $\mathcal L_0$ with a triple point or a limit triple point. 

However, since $A$ and $A'$ are two general cubic surfaces in a $\mP^3$, they have no  hyperplane  section with a triple point (having a plane section with a triple point has codimension 1 in the moduli space of smooth cubics in $\mP^3$, see, e.g., \cite [Thm. 1.3] {Cosk}). Hence there are no sections of $\mathcal L_0$ that have a triple point on $A$ or $B$.  So the triple points of the curves in $\mathcal H_t$ for $t\in \mathbb D\setminus \{0\}$, would tend to points on the double curve $R$ and would be limit triple points as described in Proposition \ref {prop:lim}. Now, if $p\in R=A\cap A'$ is a general point, there is a unique hyperplane section of $A\cup A'$ such that its two components $C_A$ and $C_{A'}$ on $A$ and $A'$ respectively, have a double point at $p$. The hyperplane in question is the one containing the tangent planes to $A$ and $A'$ at $p$, that intersect along the tangent line to $R$ at $p$. Then the two curves $C_A$ and $C_{A'}$  do have a double point at $p$, but, if $p\in R$ is general, the branches of $C_A$ and $C_{A'}$ are not tangent to $R$. This implies that if $p\in R=A\cap B$ is a general point, there is no section of $\mathcal L_0$ that has in $p$ a limit triple point. In conclusion, in the limit, we cannot have a positive dimensional family of  sections of $\mathcal L_0$ on $\mathcal X_0$ with a triple point or limit triple point, and this proves the assertion.
\end{proof} 

Next we want to prove that the set of special planes for  a general K3 surface of degree 6 and genus 4 in $\mP^4$ is non--empty and we want to compute their (expected) number. In order to do so, we need some more preliminaries. 

Let $V$ be a smooth, irreducible variety of dimension $k$ endowed with a line bundle $L$. For any integer $n\geq 0$, let $J^n(V, L)$ be the \emph{$n$--th jet bundle} relative to the pair $(V,L)$, that is defined in the following way (as a general reference on jet bundles see \cite {S}). Consider $V\times V$ with the two projections $p_i: V\times V\to V$, $1\leq i\leq 2$, to the first and second factor. One defines
$$
J^n(V, L)=p_{1*}(\mathcal O_{V\times V}/\mathcal I^n_{\Delta,V\times V} \otimes p_2^*(L))
$$
where $\Delta$ is the diagonal in $V\times V$.  Note that $J^0(V, L)=L$ and more generally $J^n(V, L)$ is a vector bundle of rank ${{n+k}\choose k}$. For all positive integers there is an exact sequence
\begin{equation}\label{eq:seq}
0\longrightarrow L \otimes {\rm Sym}^n(\Omega_V)\longrightarrow J^n(V, L) \longrightarrow J^{n-1}(V, L)\longrightarrow 0.
\end{equation}
So, at a local level, we have
$$
J^n(V, L)\sim L\otimes (1\oplus \Omega_V \oplus \cdots \oplus {\rm Sym}^n(\Omega_V))
$$
and there is an evaluation map of vector bundles
$$
{\rm ev}^n_{V,L}: H^0(V, L)\otimes \mathcal O_V \longrightarrow J^n(V, L). 
$$

Let now $\Sigma\subset \mP^4$ be a smooth, irreducible, projective, non--degenerate surface. We will denote by $H_\Sigma$ the hyperplane bundle class of $\Sigma$. There is an \emph{expected number} $\tau_\Sigma$ of hyperplane sections of $\Sigma$ having  a triple point. One has: 

\begin{prop}\label{prop:exp} If $\Sigma\subset \mP^4$ is a smooth, irreducible, projective, non--degenerate, linearly normal surface, we have
\begin{equation}\label{eq:tau}
\tau_\Sigma=5K^2_\Sigma+20 H_\Sigma\cdot K_\Sigma+15 H_\Sigma^2+5e(\Sigma)
\end{equation}
where $e(\Sigma)$ is the Euler--Poincar\'e characteristic of $\Sigma$. 
\end{prop}

\begin{proof} Consider the evaluation map
$$
{\rm ev}^2_{\Sigma, H_\Sigma}: H^0(\Sigma, H_\Sigma)\otimes \mathcal O_\Sigma \longrightarrow J^2(\Sigma, H_\Sigma), 
$$
where $h^0(\Sigma, H_\Sigma)=5$ and the rank of $J^2(\Sigma, H_\Sigma)$ is 6. The points $p\in \Sigma$ where ${\rm ev}^2_{\Sigma, H_\Sigma}$ does not have maximal rank 5, are exactly the points $p$ for which there is a section $s\in H^0(\Sigma, H_\Sigma)$ such that the curve $C$ defined by $s=0$ has a point  of multiplicity at least 3  at $p$. So the required number $\tau_\Sigma$ is nothing but the second Chern class $c_2(J^2(\Sigma, H_\Sigma))$. Using the exact sequence \eqref {eq:seq}, we have the identity of Chern polynomials
\begin{equation}\label{eq:id}
c_t(J^2(\Sigma, H_\Sigma))=c_t(H_\Sigma)\cdot c_t(\Omega_\Sigma\otimes H_\Sigma) \cdot c_t({\rm Sym}^2(\Omega_\Sigma)\otimes H_\Sigma).
\end{equation}
One has
$$
\begin{aligned}
&c_t(H_\Sigma)=1+H_\Sigma t\cr
&c_t(\Omega_\Sigma\otimes H_\Sigma)=1+(K_\Sigma+2H_\Sigma)t+(e(\Sigma)+K_\Sigma \cdot H_\Sigma+H_\Sigma^2)t^2\cr
&c_t({\rm Sym}^2(\Omega_\Sigma)\otimes H_\Sigma)=1+3(K_\Sigma+H_\Sigma)t+(2K_\Sigma^2+4e(\Sigma)+3H_\Sigma^2+6K_\Sigma\cdot H_\Sigma)t^2.
\end{aligned}
$$
Plugging into \eqref {eq:id} and computing, one checks that $\tau_\Sigma=c_2(J^2(\Sigma, H_\Sigma))$ is given by \eqref {eq:tau}.\end{proof}

\begin{rmk}\label{rmk:gen} One can easily prove a more general result by letting $\Sigma$ be not necessarily linearly normal and with improper double points. We do not dwell on this here.
\end{rmk}

\begin{cor}\label{cor:k3} Let $S$ be a general K3 surface of degree 6 and genus 4 in $\mP^4$. Then $\tau_S=210$, i.e.,  the expected number of  special planes for $S$ is 210. 
\end{cor} 

\begin{proof} It follows from Proposition \ref {prop:exp} by taking into account that $K_S=0$, $H_S^2=6$ and $e(S)=24$. \end{proof}

\section{The degree}\label{sec:degre}

Let again $S$ be a general K3 surface of degree 6 and genus 4 in $\mP^4$. In this section we will compute the degree of $X_S$ as a subvariety of  $\mathbb G(2,4)\subset \mP^9$. 
The result is the following:

\begin{thm}\label{thm:deg} If $S$ is a general K3 surface of degree 6 and genus 4 in $\mP^4$, then  the degree of $X_S$  as a subvariety of in $\mathbb G(2,4)\subset \mP^9$  is 152.
\end{thm}

The proof consists of various steps that we will perform separately.

\subsection{Schubert cycle computations} The Chow ring of $\mathbb G(2,4)$ in dimension 3 is generated by two Schubert cycles, precisely
$$
\begin{aligned}
&{\bf  h}= \{ \text{planes contained in a fixed hyperplane}\},\cr
&{\bf  k}= \{ \text{planes intersecting a fixed plane in a line passing through a fixed point}\}.
\end{aligned}
$$
Hence in the Chow ring we have the relation
$$
X_S=a{\bf h}+b{\bf k}
$$
with $a,b$ suitable integers. One has
$$
{\bf h}^2={\bf k}^2=1,\quad {\bf h}\cdot {\bf k}=0.
$$
Hence 
$$
a=X_S\cdot {\bf h}=120
$$
because $X_S\cdot {\bf h}$  equals the number of tritangential planes to $S$ contained in a general hyperplane, and this is the number of odd theta--characteristics of the curve of genus 4 general hyperplane section of $S$. 

The number
$$
b=X_S\cdot {\bf k}
$$
is more complicated to compute, namely it is the number of (proper) \emph{tritangent lines}\footnote {A (proper) tritangent [resp. bitangent, resp. tangent]  line $\ell$ to a surface $\Sigma$ in $\mP^3$ is such that $\ell$ is not contained in $\Sigma$ and $\Sigma$ cuts out on $\ell$ a divisor of the form $2D+D'$, with $D, D'$ effective divisors, $D$ of degree 3 [resp. of degree 2, resp. of degree 1] whose support lies in the smooth locus of $\Sigma$.}  to the projection $S'$ of $S$ to a $\mP^3$ from a general point, intersecting a general line of this $\mP^3$.  
 The closure of the set of proper  tritangent lines to $S'$ is a ruled surface $\mathfrak Z_S$  in $\mP^3$ and $b$ is the degree of $\mathfrak Z_S$. 

Let  $ H_\mathbb G$ be the hyperplane class of $\mathbb G(2,4)$, that is represented by the Schubert cycle
$$
H_\mathbb G=\{ \text{planes intersecting a fixed line} \}.
$$
One has clearly $H_\mathbb G^3\cdot {\bf h}=1$ and $H_\mathbb G^3\cdot {\bf k}=2$. Hence:

\begin{lem}\label{lem:partial} One has
$$\deg(X_S)=120+2\deg (\mathfrak Z_S).$$
\end{lem}

\subsection{The degree of $\mathfrak Z_S$}

To finish the proof of Theorem \ref {thm:deg} we need to compute $\deg (\mathfrak Z_S)$. Before doing that we need the following preliminary information, which can be deduced from \cite  [Prop. 3.9]{M}:

\begin{prop}\label{prop:sext} If $\tilde S\subset \mP^3$ is a general sextic surface, the scroll of tritangent lines to $\tilde S$ has degree $624$.
\end{prop}

Next we can prove the following:

\begin{prop} \label{prop:sextic} Let  $S$ be a general K3 surface of degree 6 and genus 4 in $\mP^4$. Then $\deg (\mathfrak Z_S)=16$. 
\end{prop}

\begin{proof} Let $S'$ be the projection of $S$ from a general point to $\mP^3$. Note that $S'$ is singular, having a canonical curve $\Gamma$ of degree 6 and genus 4 of double points that are generically normal crossings with finitely many pinch points. The pull back  $\tilde \Gamma$ of $\Gamma$ on $S$ is a smooth curve in $|\mathcal O_S(2)|$ and the projection map $\tilde \Gamma\longrightarrow \Gamma$ is a double cover. Since $\tilde \Gamma$ has genus $13$, by Riemann--Hurwitz formula $\tilde \Gamma\longrightarrow \Gamma$  has 12 branch points, hence there are exactly 12 pinch points of $S'$ along $\Gamma$. 

Recall that $\mathfrak Z_S$ is the scroll of (proper) tritangent lines to $S'$. Consider now a general surface $\tilde S$ of degree 6 in $\mP^3$ and consider the general specialization of $\tilde S$ to $S'$. In this specialization the scroll $\mathfrak Z_{\tilde S}$ of tritangent lines to $\tilde S$ specializes to the union of the following surfaces:\\ \begin{inparaenum}
\item [(i)] the scroll $\mathfrak Z_S$ of (proper) tritangent lines to $S'$;\\
\item [(ii)] the scroll $\mathfrak Z_1$ of (proper) bitangent lines  to $S'$ intersecting $\Gamma$ to be counted with a multiplicity $m_1$;\\
\item [(ii)] the scroll $\mathfrak Z_2$ of (proper) tangent lines  to $S'$ that are also secant to  $\Gamma$ to be counted with a multiplicity $m_2$;\\
\item [(iii)] the scroll $\mathfrak Z_3$ of trisecant lines to $\Gamma$ to be counted with a multiplicity $m_3$.
\end{inparaenum}

So we  have
\begin{equation}\label{eq:summa}
624=\deg(\mathfrak Z_{\tilde S})=\deg (\mathfrak Z_S)+\sum_{i=1}^3 m_i\deg (\mathfrak Z_i).
\end{equation}
We will compute $m_i$ and $\deg (\mathfrak Z_i)$ for all $1\leq i\leq 3$, and this will enable us to compute $\deg (\mathfrak Z_S)$. \medskip

\noindent {\bf Step 1}: computation of $\deg (\mathfrak Z_1)$ and of the multiplicity $m_1$.\medskip

The closure ${\rm Bit}(S')$ of the set of (proper) bitangent lines to $S'$ can be seen as a surface in $\mathbb G(1,3)\subset \mP^5$. The set ${\rm Int}(\Gamma)$ of lines intersecting $\Gamma$ can be seen as a hypersurface on $\mathbb G(1,3)$, namely the \emph{Cayley variety} of $\Gamma$,  and it is well known that ${\rm Int}(\Gamma)$ is cut out on $\mathbb G(1,3)$ by a hypersurface of degree $6=\deg (\Gamma)$ of $\mP^5$. Then $\mathfrak Z_1$ is contained in the intersection of ${\rm Int}(\Gamma)$ with ${\rm Bit}(S')$. More precisely we have
\begin{equation}\label{eq:for}
{\rm Int}(\Gamma) \cdot {\rm Bit}(S')=\mathfrak Z_1+\mathfrak D_1+\mathfrak D_2
\end{equation}
where:\\
\begin{inparaenum} 
\item [$\bullet$]  $\mathfrak D_1$ is a perhaps non--reduced curve supported on the curve $D_1$ parametrizing the  lines that intersect $\Gamma$ at some point $x$, are tangent to a branch of $S'$ passing through $x$ and are also a proper tangent to $S'$ at some point $y\neq x$;\\
\item [$\bullet$]  $\mathfrak D_2$ is a perhaps non--reduced curve supported on the irreducible curve $D_2$ parametrizing the tangent lines to $\Gamma$.\\
\end{inparaenum}

We will look at $D_1$ and $D_2$ as a scrolls in $\mP^3$.

\begin{cla}\label{cla:b} The surface ${\rm Bit}(S')\subset \mP^5$ has degree 132, hence $\deg ({\rm Int}(\Gamma) \cdot {\rm Bit}(S'))= 792$.
\end{cla}

\begin{proof}[Proof of Claim \ref {cla:b}] The Chow ring of $\mathbb G(1,3)$ in dimension 2 is generated by the two Schubert cycles:\\ \begin{inparaenum}
\item [$\bullet$] $P_1$ the set of lines passing through a fixed point of $\mP^3$;\\
\item [$\bullet$] $P_2$ the set of lines contained in a fixed plane of $\mP^3$;\\
\end{inparaenum}
and $P_i^2=1$ for $1\leq i\leq 2$ whereas $P_1\cdot P_2=0$.

We have the following relation in the Chow ring
$$
{\rm Bit}(S')= c P_1+d P_2
$$
with  $c,d$ integers, that we now will compute. The integer  $c$  is the number of (proper) bitangent lines to $S'$ that pass through a general point of $\mP^3$. There is a formula that computes this number (see \cite [p. 182]{Enr} and  \cite {Iv} for a modern version), that gives $c=60$.  The integer  $d$  is the number of (proper) bitangent lines to $S'$ contained in a general plane of $\mP^3$, i.e., the number of (proper) bitangent lines to a general plane section of $S'$ that is a plane sextic curve with 6 nodes. This number is provided by Pl\"ucker formulae and it turns out that  $d=72$. Since $P_1$ and $P_2$ are two planes in $\mathbb G(1,3)$, one has $\deg ({\rm Bit}(S))=60+72=132$ as wanted. Then
$\deg ({\rm Int}(\Gamma) \cdot {\rm Bit}(S'))= 6\cdot 132=792$. 
\end{proof}

\begin{cla}\label{cla:z1} One has $\deg(D_1)=612$. 
\end{cla}

\begin{proof}[Proof of Claim \ref {cla:z1}] Let $\ell$ be a general line of $\mP^3$. The degree of $D_1$ is the number of intersection points of $\ell$ with the scroll $D_1$. To compute it we proceed as follows. 

Consider the correspondence $\omega$ on $\ell$ defined in the following way. 
Take a general point $p\in \ell$. Consider all lines $r$ passing through $p$, incident to $\Gamma$ at a point $x_r$ and such that $r$ is tangent to a branch of $S'$ passing through $x_r$. Such a line intersects $S'$ off $x_r$ at three further points $x_{r,1}, x_{r,2}, x_{r,3}$. Take the tangent planes $\pi_{r,1}, \pi_{r,2}, \pi_{r,3}$ to $S'$ at $x_{r,1}, x_{r,2}, x_{r,3}$ respectively, and consider the points $p_{r,1}, p_{r,2}, p_{r,3}$ intersection of $\ell$ with $\pi_{r,1}, \pi_{r,2}, \pi_{r,3}$ respectively. The correspondence $\omega$ maps $p$ to the divisor
$$
E_p=\sum_{r} p_{r,1}+ p_{r,2}+ p_{r,3}
$$
where the sum is taken over all lines $r$ passing through $p$, incident to $\Gamma$ at a point $x_r$ and such that $r$ is tangent to a branch of $S'$ passing through $x_r$. 

Let $[e,f]$  be the \emph{indices} of $\omega$, i.e., $e=\deg(E_p)$  and $f$   is the degree of the divisor associated to a general point of $\ell$ in the inverse correspondence $\omega^{-1}$. A point 
$z\in \ell$ sits in $D_1$ if and only if it is a \emph{coincidence point} of $\omega$, i.e., if and only if $\omega(z)=\omega^{-1}(z)$. The number of coincidence points of a correspondence of indices $[i,j]$ is clearly $i+j$, hence we have
\begin{equation}\label{eq:z1}
\deg(D_1)= e+f. 
\end{equation}
So we have to compute  $e$ and $f$.

As for the computation of  $e$  we argue as follows. Take $p\in \ell$ as above. Consider the polar surface $S'_p$ of $S'$ with respect to $p$, which, by the generality of $\ell$ and $p$ is a general polar of $S'$. Then $S'_p$ contains $\Gamma$ with multiplicity 1, and intersects $S'$ in a further curve $C_p$ of degree 18 that is well known to pass simply  through the 12 pinch points of $S'$ along $\Gamma$. The curve $C_p$ is the closure of the set of smooth points $z$ of $S'$ such that the tangent plane to $S'$ at $z$ contains $p$.
On $S$ the curve $C_p$ pulls back to a curve in $|\mathcal O_S(3)|$ and $\Gamma$ to a curve in $|\mathcal O_S(2)|$. So they intersect in $36$ points. Hence $C_p$ and $\Gamma$ intersect in 24 points off the 12 pinch points of $S'$ on $\Gamma$. These points $x$ are such that the line $r=\langle x,p\rangle$ is tangent to a branch of $S'$ passing through $x$. So this proves that 
\begin{equation}\label{eq:a}
e =24\cdot 3=72. 
\end{equation}

As for the computation of  $f$  we proceed as follows. Take a general point $q\in \ell$ and consider the curve $C_q$ as above. Consider then the scroll $T$ of all lines intersecting $\ell$, $\Gamma$ and $C_q$. Some properties of $T$ are known by \emph{Salmon--Cayley's theorem} (see \cite [p. 506--509]{Cam}): $T$ has degree 180, $\ell$ has multiplicity 72, $\Gamma$ has multiplicity 18 and $C_q$ has multiplicity 6 for $T$.  

Let us consider the minimal desingularization $g: T'\longrightarrow T$ of $T$ and let $\tilde\ell$, $\tilde\Gamma$, $\tilde C_q$ be the pull--back of $\ell, \Gamma$ and $C_q$ on $T'$. It is clear that $\tilde\ell$, $\tilde\Gamma$, $\tilde C_q$ are unisecant to the lines of the ruling, whose numerical class we denote by $F$. We will denote by $H_T$ the pull back to $T'$ of the hyperplane class of $T$. 

Let us look at the pull--back on $T'$ of the intersection of $S'$ with $T$. We have
$$
6H_T\sim 2\tilde\Gamma+\tilde C_q+A
$$
where $\sim$ stays for linear equivalence and $A$ is a suitable curve. We have
$$
6=6H_T\cdot F=2\tilde\Gamma\cdot F+\tilde C_q\cdot F+A\cdot F=2+1+A\cdot F
$$
hence $A\cdot F=3$. Moreover, since $\tilde \ell \cdot \tilde\Gamma=\tilde \ell \cdot \tilde C_q=0$, one has
$$
A\cdot \tilde \ell=6H_T\cdot \tilde \ell=6\cdot 72=432.
$$
Furthermore we have
$$
\tilde\Gamma\equiv \tilde\ell+\alpha F
$$
where $\equiv$ stays for numerical equivalence and $\alpha$ is a suitable integer. Then
$$
6\cdot 18=H_T\cdot \tilde \Gamma= H_T\cdot \tilde \ell+\alpha=72+\alpha
$$
hence $\alpha=36$. So
$$
A\cdot \tilde \Gamma=A\cdot \tilde \ell+ 36(A\cdot F)= 432+ 3\cdot 36=540.
$$
Now notice that for the image $z$ on $\Gamma$ of each intersection point of $A$ with $\tilde \Gamma$, the line of the ruling of $T$ containing $z$ is tangent to a branch of $S'$ passing through $z$ and intersects $C_q$ in a point $y$ such that the tangent plane to $S'$ at $y$ passes through $q$. This proves that 
\begin{equation}\label{eq:b}
 f =540.
\end{equation}
The assertion follows by  \eqref {eq:z1},\eqref {eq:a} and \eqref {eq:b}.\end{proof}

\begin{cla}\label{cla:z1_0} One has $\deg(D_2)=18$. 
\end{cla}

\begin{proof}[Proof of Claim \ref {cla:z1_0}] This is an immediate consequence of Riemann--Hurwitz formula. \end{proof}

\begin{cla}\label{cla:z1_1} The curve $\mathfrak D_1$ is reduced, so it coincides with $D_1$. The curve  $\mathfrak D_2$ is $D_2$ counted with multiplicity 4.
\end{cla}

\begin{proof}[Proof of Claim \ref {cla:z1_1}] Let us take $x\in S'$ a general point and let $\pi_x$ be the tangent plane to $S'$ at $x$. The section of $S'$ with $\pi_x$ is a curve $C_x$ of degree 6 with six nodes, one at $x$ and six more at the intersection points of $\pi_x$ with $\Gamma$. Hence $C_x$ has geometric genus 3, and the proper bitangent lines to $S'$ passing through $x$ correspond to the $12$ ramification points of the $g^1_4$ pull--back  on the normalization $\tilde  C_x$ of $C_x$ of the series cut out  on $C_x$ by the lines of $\pi_x$ passing through $x$. 

Let us now see what happens when $x$ tends to a general point $\xi$ of $\Gamma$ on a branch $\sigma$ of $S'$ passing through $\xi$. The plane $\pi_x$ tends to the tangent plane $\pi_{\xi,\sigma}$ to the branch  $\sigma$ at $\xi$, that is also a tangent plane to $\Gamma$ at $\xi$. This plane cuts out on $S'$ a curve $C_{\xi,\sigma}$, still of geometric genus 3, that has an ordinary triple point at $\xi$ and four more nodes at the further intersection points of $\pi_{\xi,\sigma}$ with $\Gamma$. Notice that the triple point of $C_{\xi,\sigma}$ at $\xi$ has a distinguished branch $\gamma$ that is the intersection of $\pi_{\xi,\sigma}$ with the other branch $\sigma'$ of $S'$ at $\xi$. If we remove from the triple point at $\xi$ the branch $\gamma$, one is left with a node, that is cut out on the branch $\sigma$ by the tangent plane $\pi_{\xi,\sigma}$ at $\xi$. 

Now let us look at the limit, when $x$ tends to $\xi$ as indicated above,  of the $g^1_4$ on the normalization $\tilde  C_x$ of $C_x$ pull--back of the series cut out on $C_x$ by the lines of $\pi_x$ passing through $x$. The curve $\tilde  C_x$ flatly tends to the normalization $\tilde C_{\xi,\sigma}$ of $C_{\xi,\sigma}$. The $g^1_4$ on $\tilde  C_x$ tends to the $g^1_3$ pull--back on $\tilde C_{\xi,\sigma}$ of the linear series cut out on $C_{\xi,\sigma}$ by the lines of $\pi_{\xi,\sigma}$ passing through $\xi$, plus a fixed point, namely the point $y$ of $\tilde C_{\xi,\sigma}$ corresponding to the branch $\gamma$ of $C_{\xi,\sigma}$.  It is then easily seen (see, e.g.,  \cite[p. 179]{EC}) that the $12$ ramification points of the $g^1_4$ tend to the 10 ramification points of the $g^1_3$ on $\tilde C_{\xi,\sigma}$ plus the divisor $2y$, and this implies the assertion. Indeed it is clear from the above argument that $\mathfrak D_1$ is reduced. As for $\mathfrak D_2$, we see that the general point $\xi$ of $D_2$ counts in $\mathfrak D_2$ with multiplicity two for each of the two branches of $S'$ passing through $\xi$, and so it counts in $\mathfrak D_2$ with multiplicity 4. \end{proof}

\begin{cla}	\label{cla:zk11} One has $\deg(\mathfrak Z_1)=108$.
\end{cla} 

\begin{proof}[Proof of Claim \ref {cla:zk11}] The assertion follows by \eqref {eq:for} and by Claims \ref {cla:z1}, \ref {cla:z1_0} and \ref {cla:z1_1}. \end{proof}

Finally we have the:

\begin{cla}\label{cla:m1} One has $m_1=2$. \end{cla}

\begin{proof}[Proof of Claim \ref {cla:m1}] The problem is of a local nature and it boils down to the following. Let $\Sigma$ be a surface in $\mP^3$ with a double curve $D$ along which it has generically normal crossings. One has to compute the multiplicity of a line intersecting $\Sigma$ in a point of $D$ as a point of the  closure of the set of proper  tangent lines to $\Sigma$. Since the problem is local we can argue as follows. Consider a smooth quadric $Q$ in $\mP^3$ that specializes to the union of two distinct planes intersecting along a line $s$. The family of tangent lines to $Q$ is easily seen to be  the intersection of $\mathbb G(1,3)$ with a hypersurface of degree 2 in $\mP^5$.  Under the above specialization this specializes to the set of lines intersecting the line $s$, which is the intersection of $\mathbb G(1,3)$ with a hyperplane. Hence we see that  this has to be counted with multiplicity 2. The assertion follows.  \end{proof}\medskip

\noindent {\bf Step 2}: computation of $\deg (\mathfrak Z_2)$ and of the multiplicity $m_2$.\medskip

\begin{cla}\label{cla:z2} One has $\deg (\mathfrak Z_2)=90$.
\end{cla}

\begin{proof}[Proof of Claim \ref {cla:z2}]
The set of secant lines to $\Gamma$ can be considered as a surface $S(\Gamma)$ in $\mathbb G(1,3)\subset \mP^5$. The surface $S(\Gamma)$ is singular. Indeed $\Gamma$ is contained in a unique smooth quadric $Q$, that has  two rulings of lines, and the lines of the two rulings are 3--secant to $\Gamma$. Hence the two rulings in question map on $\mathbb G(1,3)$ to two irreducible conics that are sets of triple points for $S(\Gamma)$, and $S(\Gamma)$ is otherwise smooth. The normalization of $S(\Gamma)$ is smooth, it is isomorphic to the 2--fold symmetric product $\Gamma[2]$ of $\Gamma$ and we have the normalization morphism $\nu: \Gamma[2]\longrightarrow S(\Gamma)$. Let us denote by $H_{S(\Gamma)}$ the pull--back on $\Gamma[2]$ via $\nu$ of a   hyperplane section of $S(\Gamma)$. Let $r$ be a general line in $\mP^3$ and consider the $g^1_6$ cut out on $\Gamma$ by the planes containing $r$. Then  we can consider $H_{S(\Gamma)}$ as consisting of all divisors of degree 2 contained in a divisor of this $g^1_6$. It follows that
$$
H_{S(\Gamma)}^2=21, \quad p_a(H_{S(\Gamma)})=22
$$
(see \cite [Prop. (1.3)]{Cil}). 

Take now a general line $\ell$ in $\mP^3$ and consider the scroll $\Theta$ of lines that are secant to $\Gamma$ and intersect $\ell$. The surface $\Theta$ is singular.  In particular, it is immediate that $\ell$ has multiplicity 6  and $\Gamma$ has multiplicity 5 for $\Theta$. We will consider the minimal desingularization $h:\Theta'\longrightarrow \Theta$ of $\Theta$ and we will denote by $\tilde \ell$ and by $\tilde \Gamma$ the proper transform of $\ell$ and $\Gamma$ on $\Theta'$ via $h$. We will also denote by $F$ the numerical class of a line of the ruling of $\Theta'$. Then it is clear that
$$
F\cdot \tilde \ell=1, \quad F\cdot \tilde \Gamma=2. 
$$
Moreover, the ruling of $\Theta$ is birational to  $H_{S(\Gamma)}$, hence its genus is 22, so that
\begin{equation}\label{eq:12}
p_a(\tilde \ell)=22.
\end{equation}
Let $H_\Theta$ be the pull--back on $\Theta'$ of a general plane section of $\Theta$. It is clear that
$$
H_\Theta^2=H_{S(\Gamma)}^2=21.
$$
By taking the pull--back on $\Theta'$ of a general plane section of $\Theta$ containing $\ell$, we see that
$$
H_\Theta\equiv \tilde \ell+15F
$$
and therefore
$$
6=H_\Theta\cdot \tilde \ell=\tilde \ell^2+15
$$
hence 
$$
\tilde \ell^2=-9.
$$
One has 
$$
K_{\Theta'}\equiv -2\tilde \ell+\alpha F
$$
with $\alpha$ a suitable integer, and, by \eqref {eq:12}, we have
$$
42=\tilde \ell \cdot (\tilde \ell+K_{\Theta'})= \tilde \ell (-\tilde \ell +\alpha F)=9+\alpha
$$
so that $\alpha=33$.

Recall that $\Gamma$ sits on a unique quadric $Q$. By pulling--back to $\Theta'$ the intersection of $Q$ with $\Theta$, we see that
$$
2H_\Theta\equiv \tilde \Gamma+\beta F
$$
with $\beta$ a suitable integer. Intersecting with $H_\Theta$ we have
$$
42=2H_\Theta^2=H_\Theta\cdot \tilde \Gamma+\beta=30+\beta
$$
so that $\beta=12$. Hence
$$
60=2H_\Theta\cdot \tilde \Gamma= \tilde \Gamma^2+12\tilde \Gamma \cdot F=\tilde \Gamma^2+24
$$
thus 
$$
\tilde \Gamma^2=36.
$$
By taking the pull--back to $\Theta'$ of the intersection of $S'$ with $\Theta$, we have
\begin{equation}\label{eq:sum}
6H_\Theta\sim 2\tilde \Gamma+A
\end{equation}
where $A$ is a suitable curve. Intersecting with $F$ we see that $A\cdot F=2$.
Intersecting with $\tilde \Gamma$ we get
$$
180=6H_\Theta\cdot \tilde \Gamma=2\tilde \Gamma^2+A\cdot \tilde \Gamma=72+A\cdot \tilde \Gamma
$$
so that
$$
A\cdot \tilde \Gamma=108.
$$

Now
$$
2p_a(6H_\Theta)-2=6H_\Theta\cdot (6H_\Theta+K_{\Theta'})=6H_\Theta\cdot (6H_\Theta-2\tilde \ell+33F)=882
$$
so that
$$
p_a(6H_\Theta)=442.
$$
Similarly
$$
2p_a(2\tilde \Gamma)-2=2\tilde \Gamma\cdot (2\tilde \Gamma+K_{\Theta'})=2\tilde \Gamma\cdot (2\tilde \Gamma-2\tilde \ell+33F)=276
$$
so that
$$
p_a(2\tilde \Gamma)=139.
$$
From \eqref {eq:sum} we deduce
$$
442=p_a(6H_\Theta)=p_a(2\tilde \Gamma)+p_a(A)+2\tilde \Gamma\cdot A-1=139+p_a(A)+216-1
$$
so that
$$
p_a(A)=88.
$$

Recall now that we have the 2:1 morphism $A\longrightarrow \tilde \ell$ induced by the ruling of $\Theta'$. By Hurwitz formula we see that the lines of the ruling tangent to $A$ are exactly 90 and this is nothing else than the degree of $\mathfrak Z_2$.\end{proof}

\begin{cla}\label{cla:m2} One has $m_2=4$. 
\end{cla}

\begin{proof}[Proof of Claim \ref {cla:m2}] The problem being local, the assertion follows from the argument in the proof of Claim \ref {cla:m1}.\end{proof}\medskip

\noindent {\bf Step 3}: computation of $\deg (\mathfrak Z_3)$ and of the multiplicity $m_3$.\medskip

\begin{cla}\label{cla:z3} The image of $\mathfrak Z_3$ on $\mathbb G(1,3)$ consists of two disjoint conics. Moreover $m_3=8$.
\end{cla} 

\begin{proof}[Proof of Claim \ref {cla:z3}] Any trisecant line to $\Gamma$ is a line of one of the two rulings of the unique smooth quadric $Q$ containing $\Gamma$. The first assertion immediately follows. As for the second assertion, the problem being local, the assertion follows from the argument in the proof of Claim \ref {cla:m1}.\end{proof}\medskip

We are finally in a position to finish the proof of Proposition \ref {prop:sextic}. Indeed, it follows right away  from \eqref {eq:summa} and from Claims \ref {cla:zk11}, \ref {cla:m1}, \ref {cla:z2}, \ref {cla:m2}, \ref {cla:z3}. \end{proof}

Finally we have:

\begin{proof}[Proof of Theorem \ref {thm:deg}] It follows  from Proposition \ref {prop:sextic} and Lemma \ref {lem:partial}. 
\end{proof}

\end{document}